\documentclass[12pt]{amsart}
\makeindex
\newtheorem{theorem}{Theorem}[section]

\newtheorem{lemma}[theorem]{Lemma}
\theoremstyle{proposition}

\newtheorem{corollary}[theorem]{Corollary}

\newtheorem{definition}[theorem]{Definition}
\newtheorem{example}[theorem]{Example}

\theoremstyle{remark}
\usepackage[top=1.6in, bottom=1.6in, left=1.6in, right=1.5in]{geometry}

\usepackage{xcolor}
\newtheorem{remark}[theorem]{Remark}

\numberwithin{equation}{section}

\usepackage{times}
\usepackage{enumerate}
\usepackage{graphicx,adjustbox}
\usepackage{hyperref}
\usepackage{amsmath,amssymb}
\usepackage{amscd}
\usepackage{graphicx}
\usepackage[all]{xy}
\usepackage{mathtools}

\title[Join of affine semigroups]
{Join of affine semigroups}
\author{
Joydip Saha
\and
Indranath Sengupta
\and
Pranjal Srivastava
}
\date{}
\address{\small \rm  Stat Math Unit, Indian Statistical Institute, Kolkata, West-Bengal,700108, INDIA.} 
\email{saha.joydip56@gmail.com}
\thanks{The first author thanks NBHM, Government of India for the post-doc fellowship PDF/2019/001074.}

\address{\small \rm  Discipline of Mathematics, IIT Gandhinagar, Palaj, Gandhinagar, 
Gujarat 382355, INDIA.}
\email{indranathsg@iitgn.ac.in}
\thanks{The second author is the corresponding author. The second author thanks SERB, Government of India for the project CRG/2022/007047.}

\address{\small \rm  Discipline of Mathematics, IIT Gandhinagar, Palaj, Gandhinagar, Gujarat 382355, INDIA.}
\email{pranjal.srivastava@iitgn.ac.in}



\begin{document}

\begin{abstract}
In this paper, we study the class of affine semigroup generated by integral vectors,  
whose components are in generalised arithmetic progression and we observe that the 
defining ideal is determinantal. We also give a sufficient condition on the defining ideal of 
the semigroup ring for the equality of the Betti numbers of the 
defining ideal and those of its initial ideal. We introduce the notion of an 
affine semigroup generated by join of two affine semigroups and show that 
this affine semigroup exhibits some nice properties including Cohen-Macaulayness. 
\end{abstract}
\maketitle

\section{Introduction}

Let $\Gamma$ be an affine semigroup, i.e., a finitely generated semigroup, which for 
some $r$ is isomorphic to a subsemigroup of $\mathbb{Z}^{r}$ containing $0$. Let 
$\mathbb{K}$ be a field and $\mathbb{K}[\Gamma]$ be an affine semigroup ring with 
$\Gamma$ fully embedded in $\mathbb{N}^{r}$, i.e., subalgebras of the polynomial ring 
$\mathbb{K}[t_{1},\dots,t_{r}]$ generated by the monomials with exponents in 
$\Gamma$. Many authors have studied the properties of the affine semigroup ring 
$\mathbb{K}[\Gamma]$ from the properties of the affine semigroup $\Gamma$; see 
\cite{Affine,Jafari-Type}.
\medskip

Let $\Gamma$ be a simplicial affine semigroup in $\mathbb{N}^{r}$, minimally 
generated by $\mathbf{a}_{1},\dots,\mathbf{a}_{r+n}$, with extremal rays 
$\mathbf{a}_{1},\dots,\mathbf{a}_{r}$ (see Definition \ref{Extremal rays}). Let 
$\mathbb{K}[\Gamma]:=\oplus_{\gamma \in \Gamma}\mathbb{K}\mathfrak{t}^{\gamma}\subset \mathbb{K}[t_{1},\dots,t_{r}]$ be the semigroup ring corresponding to semigroup $\Gamma$, where $t_{1},\dots,t_{r}$ are indeterminates 
and $\mathfrak{t}^\gamma=\prod_{i=1}^{r}t_{i}^{\gamma_{i}}$. 
Let $A:=\mathbb{K}[z_{1},\dots,z_{r+n}]$ and define the 
map $\phi:\mathbb{K}[z_{1},\dots,z_{r+n}]\rightarrow \mathbb{K}[\Gamma]\subseteq \mathbb{K}[\mathfrak{t}]$, 
such that $\phi(z_{i})=\mathfrak{t}^{\mathbf{a}_{i}}, 1\leq i \leq r+n$. The defining ideal 
of the coordinate ring of $\mathbb{K}[\Gamma]$ is $I_{\Gamma} := \mathrm{Ker}(\phi)$.
\medskip

In general, Herzog \cite{Herzog-initial}proved that for any term order $\prec$ on $A$ and 
ideal $I \subset A$, the following relation holds: $\beta_{i}(A/I) \leq \beta_{i}(A/\mathrm{in}_{<}(I))$, 
for all $i\geq 0$, where $\beta_{i}(A/I)$ denotes the $i^{th}$ total 
Betti number in the minimal free resolution of $A/I$.
\medskip

It is always interesting to find, under which condition, equality holds in the above 
inequality.
A. Boocher \cite{Boocher-Sparse} proved that equality holds for the ideal generated 
by the maximal minor of the sparse generic matrix. In section 3, we give a condition 
on a Gr\"{o}bner basis of defining ideal $I_{\Gamma}$ which is sufficient for the 
equality of the Betti numbers of the $I_{\Gamma}$ and those of $\mathrm{in}_{<}(I_{\Gamma})$ 
(see Theorem \ref{CM}).
\medskip

Finding simplical affine semigroup rings with Cohen-Macaulay (respectively Gorenstein) 
property is not easy. In case of numerical semigroups, gluing of numerical semigroups 
is one useful technique for producing Gorenstein monomial curves, however, 
this tehniques fails in higher dimension. Giminez-Srinivasan \cite{gluing} 
observed that an affine semigroup obtained by gluing of two Cohen-Macaulay simplicial 
affine semigroups need not be Cohen-Macaulay. 
\medskip

For producing simplicial affine semigroups with Cohen-Macaulay property 
(respectively Gorenstein property), we introduce the technique 
``Join of affine semigroups'' (See Definition \ref{Join}). 
We prove that the ring associated with the join of simplicial affine 
semigroups is Cohen-Macaulay (respectively Gorenstein) provided that ring 
associated with original ones are Cohen-Macaulay (respectively Gorenstein). 
We also prove that the tensor product of the minimal 
graded free resolutions of the rings associated with the simplicial affine 
semigroups is a minimal graded free resolution of the ring associated with the 
join of these simplicial affine semigroups.
\medskip

In the last section of this article, we study the affine semigroups minimally 
generated by integral vectors whose components are in generalized arithmetic 
progession. Many mathematicians have studied the numerical semigroups generated 
by the arithmetic sequences and generalized arithmetic sequences 
(see \cite{Gimenez-Arithmetic, Isabel}. This motivated us to study these 
type of semigroups in higher dimension. 
Let $\mathbf{a},\mathbf{d}\in \mathbb{N}^{r}$ be two linearly independent elements 
over the field of a rational number $\mathbb{Q}$. For $k,h \in \mathbb{N}$, we 
consider a sequence 
$\mathbf{a},h\mathbf{a}+\mathbf{d},h\mathbf{a}+2\mathbf{d},\dots,h\mathbf{a}+n\mathbf{d}$, 
such that the affine semigroup  $\Gamma^{r}_{\mathbf{a},\mathbf{d},h}:=\langle \mathbf{a},h\mathbf{a}+\mathbf{d},h\mathbf{a}+2\mathbf{d},\dots,h\mathbf{a}+n\mathbf{d} \rangle$ is minimally generated by this sequence. In this article, $\Gamma^{r}_{\mathbf{a},\mathbf{d},h}$ 
denotes an affine semigroup in $\mathbb{N}^{r}$ of embedding dimension $n+1$. 
Recently, Bhardwaj et al. \cite{Bhardwaj-arithmetic} studied the properties of $\Gamma^{2}_{\mathbf{a,d},1}$. They have found a minimal graded resolution only for $n=2,3,4$. In this article, we prove that the defining ideal of $\mathbb{K}[\Gamma^{2}_{\mathbf{a,d},1}]$ has a determinantal structure, hence a minimal graded resolution of $\Gamma^{r}_{\mathbf{a},\mathbf{d},h}$ can be computed using the Eagon-Northcott complex, for all values of $n$. 
\medskip

For $1\leq i \leq r, r \in \mathbb{N}$, let $\mathbf{a}_{i},\mathbf{d}_{i}\in  \mathbb{N}^{2r}$ 
be linearly independent elements over $\mathbb{Q}$. Consider a set $S_{i}=\{ \mathbf{a}_{i},h\mathbf{a}_{i}+\mathbf{d}_{i},\dots,h\mathbf{a}_{i}+n_{i}\mathbf{d}_{i}\}$. Assume $\{\mathbf{a}_{1},\mathbf{d}_{1},\dots,\mathbf{a}_{r},\mathbf{d}_{r}\}$ is linearly independent over $\mathbb{Q}$. The affine semigroup $\mathfrak{S}_{r}=\big\langle \cup_{i=1}^{r}\langle S_{i}\rangle\big\rangle$ is 
called the join of the semigroups $\langle S_{i}\rangle$. It is a simplicial affine semigroup with extremal rays 
$\{\mathbf{a}_{1},\mathbf{d}_{1},\dots,\mathbf{a}_{r},\mathbf{d}_{r}\}$. We prove that 
$\mathbb{K}[\mathfrak{S}_{r}]$ has several properties like Koszul, Cohen-Macaulay etc., and 
the associated graded ring $\mathrm{gr}_{\mathfrak{m}}(\mathbb{K}[\mathfrak{S}_{r}])$ is Cohen-Macaulay with regularity $2$.

\section{Preliminaries}

\begin{definition}{\rm 
Let $\Gamma$ be an affine semigroup, fully embedded in $\mathbb{N}^{r}$, minimally generated by $\mathbf{a}_{1},\dots,\mathbf{a}_{n+r}$. The \textit{rational polyhedral cone} generated by $\Gamma$ is defined as
\[
 \mathrm{cone}(\Gamma)=\big\{\sum_{i=1}^{n+r}\alpha_{i}\mathbf{a}_{i}: \alpha_{i} \in \mathbb{R}_{\geq 0}, \,i=1,\dots,r+n\big\}.
 \]
 The \textit{dimension} of $\Gamma$ is defined as the dimension of the subspace generated by $\mathrm{cone}(\Gamma)$.
 }
\end{definition} 
   
The $\mathrm{cone}(\Gamma)$ is the intersection of finitely many closed linear half-spaces in $\mathbb{R}^{r}$, 
each of whose bounding hyperplanes contains the origin. These half-spaces are called \textit{support hyperplanes}. 

\begin{definition}\label{Extremal rays} {\rm 
Suppose $\Gamma$ is an affine semigroup, fully embedded in $\mathbb{N}^{r}$. If $r=1$, $\mathrm{cone}(\Gamma)=\mathbb{R}_{\geq 0}$.
If $r=2$, the support hyperplanes are one-dimensional vector spaces, which are called the 
\textit{extremal rays} of $\mathrm{cone}(\Gamma)$. If $r >2$, intersection of any two
adjacent support hyperplanes is a one-dimensional vector space, called an extremal ray 
of $\mathrm{cone}(\Gamma)$. An element of $\Gamma$ is called an extremal ray of $\Gamma$ 
if it is the smallest non-zero vector of $\Gamma$ in an extremal ray of $\mathrm{cone}(\Gamma)$. 
}
\end{definition}

\begin{definition}\label{Extremal}{\rm 
An affine semigroup $\Gamma$, fully embedded in $\mathbb{N}^{r}$, is said to be \textit{simplicial}   
if the $\mathrm{cone}(\Gamma)$ has atleast $r$ extremal rays, i.e., if there exist $r$ elements,  say 
$\{\mathbf{a}_{1},\dots,\mathbf{a}_{r}\} \subset \{\mathbf{a}_{1},\dots,\mathbf{a}_{r},\mathbf{a}_{r+1},\dots,\mathbf{a}_{r+n}\}$, such that they 
are linearly independent over $\mathbb{Q}$ and $\Gamma \subset \sum\limits_{i=1}^{r}\mathbb{Q}_{\geq 0}\mathbf{a}_{i}$.
}
\end{definition}

In this paper, $\Gamma$ always denotes a simplicial affine semigroup minimally generated by 
$\{\mathbf{a}_{1}.\dots,\mathbf{a}_{r},\mathbf{a}_{r+1},\dots,\mathbf{a}_{r+n}\}$, with the set of extremal rays $E=\{\mathbf{a}_{1}.\dots,\mathbf{a}_{r}\}$. 
The semigroup ring defined by $\Gamma$ is written as $\mathbb{K}[\Gamma]=\mathbb{K}[\mathfrak{t}^{\mathbf{a}_{1}},\dots,\mathfrak{t}^{\mathbf{a}_{r+n}}]$.

\begin{definition}\emph{
The \textit{Ap\'{e}ry set} of $\Gamma$ with respect to an element $\mathbf{b} \in \Gamma$ is defined as $\{\mathbf{a} \in \Gamma: \mathbf{a}-\mathbf{b} \notin \Gamma\}$. Let $E=\{\mathbf{a}_{1},\dots,\mathbf{a}_{r}\}$ be a set of extremal rays of $\Gamma$, then the Ap\'{e}ry set of $\Gamma$ with respect to the set $E$ is
\[
\mathrm{AP}(\Gamma,E)=\{a \in \Gamma \mid \mathbf{a}-\mathbf{a}_{i} \notin \Gamma, \, \forall i=1,\dots,r\}=\cap_{i=1}^{r}\mathrm{AP}(\Gamma,\mathbf{a}_{i}).
\]}
\end{definition}

For an affine semigroup $\Gamma$, consider the natural partial ordering 
$\prec_{\Gamma}$ on $\mathbb{N}$, such that for all elements 
$x,y \in \mathbb{N}^{r}$, $x \prec_{\Gamma} y \,\,\, \text{if} \,\,\, y-x \in \Gamma$. 

\begin{definition}[\cite{Jafari-Type}]{\rm
Let $\mathbf{b} \in \mathrm{max}_{\prec_{\Gamma}} \mathrm{Ap}(\Gamma,E)$; the element 
$\mathbf{b}-\sum_{i=1}^{r}\mathbf{a}_{i} $ is called a \emph{quasi-Frobenius} element of $\Gamma$. 
The set of all quasi-Frobenius elements of $\Gamma$ is denoted by $\mathrm{QF}(\Gamma)$ 
and its cardinality is said to be the \emph{type} of $\Gamma$, denoted by $\mathrm{type}(\Gamma)$. 
}
\end{definition}

\begin{remark}\cite[Proposition 3.3]{Jafari-Type}
If $\mathbb{K}[\Gamma]$ is arithmetically Cohen-Macaulay, then the last Betti 
number of $\mathbb{K}[\Gamma]$ is called the Cohen-Macaulay type of 
$\mathbb{K}[\Gamma]$, written as $\mathrm{CMtype}(\mathbb{K}[\Gamma])$. 
Moreover, $\mathrm{type}(\Gamma)=\mathrm{CMtype}(\mathbb{K}[\Gamma])$.
\end{remark}

\medskip
\noindent\textbf{The Eagon-Northcott Complex.} Let us recall some basics facts 
on the Eagon-Nothcott complex from \cite{eisenbud}. Let $F=R^{f}$ and $G=R^{g}$ be free modules of finite rank 
over the polynomial ring $R$. The \textit{Eagon-Northcott complex} of a map 
$\alpha: F\longrightarrow G$ (or that of a matrix $A$ representing $\alpha$) 
is a complex 
\begin{eqnarray*}
{EN}(\alpha): 
0\rightarrow (Sym_{f-g}G)^{*}\otimes \wedge^{f}F\stackrel{d_{f-g+1}}\longrightarrow (Sym_{f-g-1}G)^{*}
\otimes \wedge^{f-1}F\stackrel{d_{f-g}} 
\longrightarrow\\
\cdots\longrightarrow(Sym_{2}G)^{*}\otimes
\wedge^{g+2}F\stackrel{d_3}\longrightarrow G^{*}\otimes 
\wedge^{g+1}F\stackrel{d_{2}}\wedge^{g}F\stackrel{\wedge^{g}\alpha}\longrightarrow \wedge^{g} G.
\end{eqnarray*} 
Here $Sym_{k}G$ is the $k$-th symmetric power of G and $M^{*} = {\rm Hom}_{R}(M,R)$. 
The map $d_{j}$ are defined as follows. First we define a diagonal map 
\begin{eqnarray*}
(Sym_{k}G)^{*} & \rightarrow & G^{*}\otimes (Sym_{k-1}G)^{*}\\
u & \mapsto & \sum_{i}u_{i}^{'}\otimes u_{i}^{''}
\end{eqnarray*}
as the dual of the multiplication map $G\otimes Sym_{k-1}G\longrightarrow Sym_{k}G$ 
in the symmetric algebra of $G$. Next we define an analogous diagonal map 
\begin{eqnarray*}
\wedge^{k}F & \longrightarrow & F\otimes \wedge^{k-1}F\\
v & \mapsto & \sum_{i}v_{i}^{'}\otimes v_{i}^{''}
\end{eqnarray*} 
as the dual of the multiplication in the exterior algebra of $F^{*}$.
\medskip

\begin{theorem}[Eagon-Northcott]\label{Eagon-Northcott} The Eagon-Northcott complex is a free resolution of $R/I_{g}(\alpha)$ iff grade$(I_{g}(\alpha))=f-g+1$
where $I_{g}(\alpha)$ denotes the $g\times g$ minors of the matrix $A$ representing 
$\alpha$.
\end{theorem}

\proof See \cite[Corollary A2.12]{eisenbud}.\qed
\bigskip

\section{Betti numbers of the initial ideal of Semigroup rings}

In this section, we give a condition on the defining ideal of the 
semigroup ring which is sufficient for the equality of the 
Betti numbers of the defining ideal and the Betti numbers of 
its initial ideal, with respect to given monomial order.

\begin{theorem}\label{CM}
Let $A/I_{\Gamma}$ be a Cohen-Macaulay. 
Let $G=\{f_{1},\dots,f_{s}\}$ be a minimal Gr\"{o}bner basis of the defining ideal $I_{\Gamma}$, 
with respect to a monomial ordering $>$. Suppose that $A/\mathrm{in}_{>}(I_{\Gamma})$ is 
Cohen-Macaulay and there exist $z_{j}$, $1\leq j \leq r$, such that $z_{j}$ belongs to the 
support of $f_{l}$ and does not divide $\mathrm{LM}(f_{l})$, 
for every $1 \leq l \leq s $, then $\,\, \beta_{i}(A/I_{\Gamma})=\beta_{i}(A/\mathrm{in}_{>}(I_{\Gamma}))$, 
for all $i \geq 1$. 
\end{theorem}

\proof 
We consider the map $$\pi:A=\mathbb{K}[z_{1},\dots,z_{d},\dots,z_{r+n}] \rightarrow \bar{A}=\mathbb{K}[z_{r+1},\dots,z_{r+n}],$$ such that $\pi(x_{j})=0$, $\, 1 \leq j \leq r$ and $\pi(z_{j})=z_{j}$, 
$\, r+1 \leq j \leq r+n$. 

It is easy to check that 
$A/(I_{\Gamma},z_{1},\dots,z_{r})\cong \bar{A}/\pi(I_{\Gamma})$. 
Since $\mathbb{K}[\Gamma]$ is Cohen-Macaulay, 
$z_{1},\dots,z_{r}$ forms a regular sequence in $A/I_{\Gamma}$, 
and the Betti numbers are preserved under the division by 
regular elements (see \cite[Lemma 3.1.16]{Bruns-Herzog}). Hence,
\begin{align*}
\beta_{i}^{A}(A/I_{\Gamma})=&\beta_{i}^{A}(A/(I_{\Gamma},z_{1},\dots,z_{r}))\\
=& \beta_{i}^{A}(\bar{A}/\pi(I_{\Gamma})).
\end{align*}

Since $z_{j}$ belongs to the support of $f_{l}$ and do not divide the $\mathrm{LM}(f_{l})$, 
for some $j$ and for all $l$, so $\pi(I_{\Gamma})=\mathrm{in}_{>}(\mathbb{K}[\Gamma])$. As $\mathrm{in}_{>}(I_{\Gamma})$ is Cohen-Macaulay, therefore
\begin{align*}
\beta_{i}^{A}(\bar{A}/\pi(I_{\Gamma}))=&\beta_{i}^{A}(\bar{A}/\mathrm{in}_{>}(I_{\Gamma}))\\
=& \beta_{i}^{A}(A/(\mathrm{in}_{>}(I_{\Gamma}),z_{1},\dots,z_{r}))\\
=&\beta_{i}^{A}(A/(\mathrm{in}_{>}(I_{\Gamma}))). \qed
\end{align*}

\begin{corollary}
Let $\mathbb{K}[\Gamma]$ be a Cohen-Macaulay. 
Let $G=\{f_{1},\dots,f_{s}\}$ be a minimal Gr\"{o}bner basis of the defining ideal $I_{\Gamma}$, with respect to the negative degree reverse monomial ordering $>$ induced by $z_{r+s} > \dots > z_{r} > \dots >z_{1}$.
If there exist $z_{j}$, $1\leq j \leq r$, such that 
$z_{j}$ belongs to the support of $f_{l}$, 
then $$\beta_{i}(\mathbb{K}[\Gamma])=\beta_{i}(\mathrm{in}_{>}(\mathbb{K}[\Gamma]))= \beta_{i}(\mathrm{gr}_{\mathfrak{m}}(\mathbb{K}[\Gamma])
\quad \forall i \geq 1.$$
\end{corollary}
\begin{proof}
Easily follows from Theorem \ref{CM} and \cite[Theorem 3.8]{Associated-Saha}.
\end{proof}

\begin{example}{\rm
Backelin defined the class of semigroups $\langle s, s+3, s+3n+1, s+3n+2\rangle$, 
for $n\geq 2$, $r\geq 3n+2$ and $s=r(3n+2)+3$. Let 
$\tilde{\Gamma}=\langle (0,s+3n+2),(s,3n+2),(s+3,3n-1),(s+3n+1,1),(s+3n+2,0)\rangle \subset \mathbb{N}^{2}$. 
It is known that $\mathbb{K}[\tilde{\Gamma}]$ is Cohen-Macaulay (see \cite[Theorem 2.9]{Backelin}. 
Note that $\{(0,s+3n+2),(s+3n+2,0)\}$ is the set of extremal rays of $\tilde{\Gamma}$ and 
$z_{1}, z_{5}$ belong to the support all elements of a minimal Gr\"{o}bner 
basis of the defining ideal of the projective closure of Backelin curves 
(see \cite[Theorem 2.5]{Backelin}), with respect to the reverse degree lexicographic 
ordering induced by $x_{1}>\dots>x_{4}>x_{0}$. 
Hence, by Theorem \ref{CM}, $\beta_{i}(A/I_{\Gamma})=\beta_{i}(A/\mathrm{in}_{>}(I_{\Gamma})) 
\quad \forall i \geq 1$.
}
\end{example}

\section{Join of Affine Semigroups}

In this section, we introduce the notion of join of affine semigroups.

\begin{definition}\label{Join}{\rm
Let $\Gamma_{1}$ and $\Gamma_{2}$ be two affine semigroups in $\mathbb{N}^{r_{1}+r_{2}}$, such 
that they are 
minimally generated by 
the disjoint sets $\mathcal{G}_{1}=\{\mathbf{a}_{1},\dots,\mathbf{a}_{r_{1}},\dots,\mathbf{a}_{r_{1}+n_{1}}\}$ and $\mathcal{G}_{2}=\{\mathbf{b}_{1},\dots,\mathbf{b}_{r_{2}},\dots,\mathbf{b}_{r_{2}+n_{2}}\}$, with extremal rays $E_{\Gamma_{1}}=\{\mathbf{a}_{1},\dots,\mathbf{a}_{r_{1}}\}$, $E_{\Gamma_{2}}=\{\mathbf{b}_{1},\dots,\mathbf{b}_{r_{2}}\}$ respectively.

Assuming that the set $\{\mathbf{a}_{1},\dots,\mathbf{a}_{r_{1}},\mathbf{b}_{1},\dots,\mathbf{b}_{r_{2}}\}$ is linearly independent over $\mathbb{Q}$, the semigroup $\Gamma_{1}\sqcup\Gamma_{2}=\langle\mathcal{G}_{1}\cup\mathcal{G}_{2}\rangle$ is a simplicial affine semigroup with the set of  extremal rays $E_{\Gamma_{1}\sqcup\Gamma_{2}}=E_{\Gamma_{1}}\cup E_{\Gamma_{2}}$. Moreover, 
the set $\mathcal{G}_{1}\cup\mathcal{G}_{2}$ is a minimal generating set of the semigroup 
$\Gamma_{1}\sqcup\Gamma_{2}$. We call $\Gamma_{1}\sqcup\Gamma_{2}$ the \textit{join} 
of the affine semigroups $\Gamma_{1}$ and $\Gamma_{2}$. 
}
\end{definition}

Let $A_{1}=\mathbb{K}[x_{1},\ldots,x_{r_{1}+n_{1}}]$ 
and $A_{2}=\mathbb{K}[y_{1},\ldots,y_{r_{2}+n_{2}}]$ be polynomial rings 
with disjoint set of 
indeterminates $\{x_{1},\ldots , x_{r_{1}+n_{1}}\}$ and $\{y_{1},\dots,y_{r_{2}+n_{2}}\}$. 
Let $\mathbb{K}[\mathbf{t}]$ be another polynomial ring, where 
$\mathbf{t}=t_{1},\ldots,t_{r_{1}+r_{2}}$. 
We consider the maps $\phi_{\Gamma_{1}}:A_{1}\rightarrow \mathbb{K}[\mathbf{t}]$,  
defined by $\phi_{\Gamma_{1}}(x_{i})=\mathbf{t}^{\mathbf{a_{i}}}$, $1 \leq i \leq r_{1}+n_{1}$, 
and 
$\phi_{\Gamma_{2}}:A_{2}\rightarrow \mathbb{K}[\mathbf{t}]$, defined by 
$\phi_{\Gamma_{1}}(y_{j})=\mathbf{t}^{\mathbf{b_{j}}}$, $1 \leq j \leq r_{2}+n_{2}$. 
We write $A_{12}=\mathbb{K}[x_{1},\dots,x_{r_{1}+n_{1}},y_{1},\dots,y_{r_{2}+n_{2}}]$ and 
consider $\phi_{\Gamma_{1}\sqcup\Gamma_{2}}: A_{12}\rightarrow \mathbb{K}[\mathbf{t}]$, 
defined by 
$\phi_{\Gamma_{1}\sqcup\Gamma_{2}}(x_{i})=\mathbf{t}^{\mathbf{a_{i}}}$, 
$1 \leq i \leq r_{1}+n_{1}$, and 
$\phi_{\Gamma_{1}\sqcup\Gamma_{2}}(y_{j})=\mathbf{t}^{\mathbf{b_{j}}}$, 
$1 \leq j \leq r_{2}+n_{2}$.

\medskip

\begin{theorem}\label{Generator-Concatanation}
Let $I_{\Gamma_{1}}, I_{\Gamma_{2}}$ be the defining ideals of $\mathbb{K}[\Gamma_{1}]$ and $\mathbb{K}[\Gamma_{2}]$ respectively. Then the defining ideal 
of $\mathbb{K}[\Gamma_{1}\sqcup\Gamma_{2}]$ is 
$I_{\Gamma_{1}\sqcup\Gamma_{2}} = I_{\Gamma_{1}}A_{12}+I_{\Gamma_{2}}A_{12}$. The 
tensor product of minimal graded free resolutions of $\mathbb{K}[\Gamma_{1}]$ and $\mathbb{K}[\Gamma_{2}]$ over $\mathbb{K}$ is a 
minimal graded free resolution of $\mathbb{K}[\Gamma_{1}\sqcup\Gamma_{2}]$ over $\mathbb{K}$.
\end{theorem}

\begin{proof}
Let $\mathbf{x}^{\mathbf{\alpha}}-\mathbf{y}^{\mathbf{\beta}}$ be an element of the minimal  generating set of $I_{\Gamma_{c}}$, where $\mathbf{\alpha}=(\alpha_{1},\dots,\alpha_{r_{1}+n_{1}}),\mathbf{\beta}=(\beta_{1},\dots,\beta_{r_{2}+n_{2}})$. Then, there exist an element $\gamma$ of $\Gamma_{c}$ such that $\gamma=\sum_{i=1}^{r_{1}+n_{1}}\alpha_{i}\mathbf{a}_{i}=\sum_{j=1}^{r_{2}+n_{2}}\beta_{j}\mathbf{b}_{j}$. 
Since $\{\mathbf{a}_{1},\dots,\mathbf{a}_{r_{1}}\}$, $\{\mathbf{b}_{1},\dots,\mathbf{b}_{r_{2}}\}$ are sets of extremal rays of $\Gamma_{1}$ and $\Gamma_{2}$, so $\mathbf{a}_{r_{1}+l_{1}}=\sum_{j=1}^{r_{1}} \alpha_{l_{1},j}\mathbf{a}_{j},  \alpha_{l_{1},j}\in \mathbb{Q},1 \leq l_{1} \leq n_{1}$ and $\mathbf{b}_{r_{2}+l_{2}}=\sum_{j=1}^{r_{2}} \beta_{l_{2},j}\mathbf{b}_{j},  \alpha_{l_{2},j}\in \mathbb{Q},1 \leq l_{2} \leq n_{2}$. 
Now, without loss of generality, assume that $\alpha_{r_{1}+n_{1}}\neq 0$. Then $\alpha_{r_{1}+n_{1}}\mathbf{a}_{r_{1}+n_{1}}=\sum_{j=1}^{r_{2}+n_{2}}\beta_{j}\mathbf{b}_{j}-\sum_{i=1}^{r_{1}+n_{1}-1}\alpha_{i}\mathbf{a}_{i}$. This implies, $\alpha_{r_{1}+n_{1}}\sum_{j=1}^{r_{1}}\alpha_{r_{1}+n_{1},j}\mathbf{a}_{j}=\sum_{j=1}^{r_{2}}\beta_{j}\mathbf{b}_{j}+\sum_{j=r_{2}+1}^{r_{2}+n_{2}}\beta_{j}(\sum_{l=1}^{r_{2}} \beta_{l_{2},l}\mathbf{b}_{l})-\sum_{i=1}^{r_{1}}\alpha_{i}\mathbf{a}_{i}-\sum_{j=r_{1}+1}^{r_{1}+n_{1}-1}\alpha_{i}(\sum_{j=1}^{r_{1}} \alpha_{l_{1},i}\mathbf{a}_{i})$, which is a contradiction because 
the set $E_{\Gamma_{1}\sqcup\Gamma_{2}}=\{\mathbf{a}_{1},\dots,\mathbf{a}_{r_{1}},\mathbf{b}_{1},\dots,\mathbf{b}_{r_{2}}\}$ is linearly independent over $\mathbb{Q}$. Similarly, by the same argument as above, any element of the form $\mathbf{x}^{\mathbf{\alpha}}\mathbf{y}^{\mathbf{\beta}}-\mathbf{x}^{\mathbf{\alpha}'}\mathbf{y}^{\mathbf{\beta}'}$ does not belong 
to a minimal generating set of $I_{\Gamma_{c}}$. Hence, all the elements of a 
minimal generating set of $I_{\Gamma_{\Gamma_{1}\sqcup\Gamma_{2}}}$ are of the form $\mathbf{x}^{\mathbf{\alpha}}-\mathbf{x}^{\mathbf{\alpha'}}\in I_{\Gamma_{1}}$ and $\mathbf{y}^{\mathbf{\beta}}-\mathbf{y}^{\mathbf{\beta'}}\in I_{\Gamma_{2}}$, and we get that 
$I_{\Gamma_{\Gamma_{1}\sqcup\Gamma_{2}}}$ is generated by $I_{\Gamma_{1}} \cup I_{\Gamma_{2}}$.
\medskip

Minimal graded free resolutions $\mathbb{M}_{1}$ of $\mathbb{K}[\Gamma_{1}]\cong A_{1}/I_{\Gamma_{1}} $ and $\mathbb{M}_{2}$ of 
$\mathbb{K}[\Gamma_{1}]\cong A_{2}/I_{\Gamma_{2}}$ 
can be seen as minimal graded free resolutions of the $A_{12}$-modules 
$A_{12}/I_{\Gamma_{1}}A_{12} $ and $A_{12}/I_{\Gamma_{2}}A_{12}$. 
The ideals $I_{\Gamma_{1}}$ and $I_{\Gamma_{2}}$ are generated by disjoint 
sets of variable in $A_{12}$, therefore, $\mathbb{M}_{1} \otimes \mathbb{M}_{2}$ is a minimal graded free resolution of $A_{12}/(I_{\Gamma_{1}}A_{12}+I_{\Gamma_{2}}A_{12})$ over 
$\mathbb{K}$.
\end{proof}

\begin{corollary}
If $\mathbb{K}[\Gamma_{1}]$ and $\mathbb{K}[\Gamma_{2}]$ are Cohen-Macaulay, then $\mathbb{K}[\Gamma_{1}\sqcup\Gamma_{2}]$ is Cohen-Macaulay.
\end{corollary}

\proof Since tensor product of minimal graded free resolution of $\mathbb{K}[\Gamma_{1}]$ and $\mathbb{K}[\Gamma_{2}]$ gives a minimal graded free resolution of $\mathbb{K}[\Gamma_{c}]$, we have $\mathrm{projdim}(\mathbb{K}[\Gamma_{1}\sqcup\Gamma_{2}])= \mathrm{projdim}(\mathbb{K}[\Gamma_{1}])+\mathrm{projdim}(\mathbb{K}[\Gamma_{2}])$. Using Auslander-Buschbaum formula, we 
can conclude that $\mathbb{K}[\Gamma_{1}\sqcup\Gamma_{2}]$ is Cohen-Macaulay.\qed

\section{The Join of generalized arithmetic sequences}

Let $\mathbf{a},\mathbf{d}\in \mathbb{N}^{r}$ be two linearly 
independent elements over the field of a rational number $\mathbb{Q}$.  In this article, 
for $k,h \in \mathbb{N}$, we consider a sequence $\mathbf{a},h\mathbf{a}+\mathbf{d},h\mathbf{a}+2\mathbf{d},\dots,h\mathbf{a}+n\mathbf{d}$ , 
such that the affine semigroup $\Gamma^{r}_{\mathbf{a},\mathbf{d},h}:=\langle \mathbf{a},h\mathbf{a}+\mathbf{d},h\mathbf{a}+2\mathbf{d},\dots,h\mathbf{a}+n\mathbf{d} \rangle$ is minimally generated by this sequence. Therefore, $\Gamma^{r}_{\mathbf{a},\mathbf{d},h}$ is an affine semigroup in $\mathbb{N}^{r}$ of embedding dimension $n+1$.
\medskip

For $1\leq i \leq r, r \in \mathbb{N}$, let $\mathbf{a}_{i},\mathbf{d}_{i}\in  \mathbb{N}^{2r}$ be linearly independent elements over $\mathbb{Q}$. Consider the set $S_{i}=\{ \mathbf{a}_{i},h\mathbf{a}_{i}+\mathbf{d}_{i},\dots,h\mathbf{a}_{i}+n_{i}\mathbf{d}_{i}\}$. Assume $\{\mathbf{a}_{1},\mathbf{d}_{1},\dots,\mathbf{a}_{r},\mathbf{d}_{r}\}$ is linearly independent over $\mathbb{Q}$. We construct the join 
$\mathfrak{S}_{r}=\big\langle \cup_{i=1}^{r}\langle S_{i}\rangle\big\rangle$, which 
is a simplicial affine semigroup with extremal rays $\{\mathbf{a}_{1},\mathbf{d}_{1},\dots,\mathbf{a}_{r},\mathbf{d}_{r}\}$.

\begin{lemma}
$\Gamma^{2}_{\mathbf{a},\mathbf{d},h}$ is a simplicial affine semigroup in $\mathbb{N}^{2}$ with respect to the extremal rays $\mathbf{a}$ and $h\mathbf{a}+n\mathbf{d}$.
\end{lemma}

\begin{proof}
Proof is easily follows from the inequality $n(h\mathbf{a}+l\mathbf{d})=(nh-lh)\mathbf{a}+l(h\mathbf{a}+n\mathbf{d})$ for $l=1,\dots,n$.
\end{proof}

\begin{theorem}\label{Apery}
Let $E=\{\mathbf{a},h\mathbf{a}+n\mathbf{d}\}$ be the set of extremal rays of $\Gamma^{2}_{\mathbf{a},\mathbf{d},h}$. The Ap\'{e}ry set of $\Gamma^{2}_{\mathbf{a},\mathbf{d},h}$, with respect to $E$, is 
$$\mathrm{Ap}(\Gamma^{2}_{\mathbf{a},\mathbf{d},h},E)=\{0,h\mathbf{a}+\mathbf{d},\dots,h\mathbf{a}+(n-1)\mathbf{d}\}.$$
\end{theorem}

\proof
Let $\mathbf{b}\in \mathrm{Ap}(\Gamma,E)$. Then $\mathbf{b}=\sum_{i=1}^{n-1}\lambda_{i}(h\mathbf{a}+i\mathbf{d})$. Suppose there exist $1 \leq s,t \leq n-1$, such that $\lambda_{s}+\lambda_{t}\geq 2$. Now, if 
\begin{enumerate}
\item[(i)]$s+t\leq n$, then $\mathbf{b}-\mathbf{a}=h\mathbf{a}+(s+t)\mathbf{d}+\sum_{i\in \{s,t\}}(\lambda_{i}-1)(h\mathbf{a}+i\mathbf{d})+\sum_{i \in [1,n-1]\setminus \{s,t\}}\lambda_{i}(h\mathbf{a}+i\mathbf{d})$;
\medskip

\item[(ii)] $s+t> n$, then $ s+t=n+n^{'}$ for some $n' \in \mathbb{N}$, and it follows that 
$\mathbf{b}-(h\mathbf{a}+n\mathbf{d})=h\mathbf{a}+n^{'}\mathbf{d}+(s+t)\mathbf{d}+\sum_{i\in \{s,t\}}(\lambda_{i}-1)(h\mathbf{a}+i\mathbf{d})+\sum_{i \in [1,n-1]\setminus \{s,t\}}\lambda_{i}(h\mathbf{a}+i\mathbf{d})$.
\end{enumerate}
\medskip

In all the cases, we get a contradiction to $\mathbf{b} \in \mathrm{Ap}(\Gamma^{2}_{\mathbf{a},\mathbf{d},h},E)$. Hence  $\mathrm{Ap}(\Gamma^{2}_{\mathbf{a},\mathbf{d},h},E) \subseteq \{0,h\mathbf{a}+\mathbf{d},\dots,h\mathbf{a}+(n-1)\mathbf{d}\}$. It is clear that 
$$\{0,h\mathbf{a}+\mathbf{d},\dots,h\mathbf{a}+(n-1)\mathbf{d}\} \subseteq \mathrm{Ap}(\Gamma^{2}_{\mathbf{a},\mathbf{d},h},E). \qed$$

\begin{theorem}\label{Apery-Concatanation}
Let $E_{c}=\{\mathbf{a}_{1},\mathbf{d}_{1},\mathbf{a}_{2},\mathbf{d}_{2}\}$ be the set of extremal rays 
of $\mathfrak{S}_{2}$. The Ap\'{e}ry set of $\mathfrak{S}_{2}$, with respect to $E_{c}$, is 
$$\mathrm{Ap}(\mathfrak{S}_{2},E_{c})=\{\lambda (h\mathbf{a}_{1}+i\mathbf{d}_{1})+ \mu (h\mathbf{a}_{2}+j\mathbf{d}_{2}) \vert \lambda,\mu\in \{0,1\}, 1 \leq i \leq n_{1}-1,1 \leq j \leq n_{2}-1\}.$$
\end{theorem}
\begin{proof}
Let $\mathbf{b}\in \mathrm{Ap}(S,E_{c})$. Then $\mathbf{b}=\sum_{i=1}^{n_{1}-1}\lambda_{i}(h\mathbf{a}_{1}+i\mathbf{d}_{1})+\sum_{j=1}^{n_{2}-1}\mu_{j}(h\mathbf{a}_{2}+j\mathbf{d}_{2})$. Suppose there exist 
$1 \leq s_{1},t_{1} \leq n_{1}-1, 1 \leq s_{2},t_{2} \leq n_{2}-1$, such that 
$\lambda_{s_{1}}+\lambda_{t_{1}}\geq 2, \mu_{s_{2}}+\mu_{t_{2}}\geq 2$. Now, if 
\begin{itemize}
\item[(i)]$s_{1}+t_{1}\leq n_{1}$, then $\mathbf{b}-\mathbf{a}_{1}=h\mathbf{a}_{1}+(s_{1}+t_{1})\mathbf{d}_{1}+\sum_{i\in \{s_{1},t_{1}\}}(\lambda_{i}-1)(h\mathbf{a}_{1}+i\mathbf{d}_{1})+\sum_{i \in [1,n_{1}-1]\setminus \{s_{1},t_{1}\}}\lambda_{i}(h\mathbf{a}_{1}+i\mathbf{d}_{i})+\sum_{j=1}^{n_{2}-1}\mu_{j}(h\mathbf{a}_{2}+j\mathbf{d}_{2})$. 
\medskip

\item[(ii)] $s_{1}+t_{1} > n_{1}$, then $ s_{1}+t_{1}=n_{1}+n_{1}^{'}$ for some $n'\in \mathbb{N}$, and 
it follows that 
\begin{align*}
&\mathbf{b}-(h\mathbf{a}_{1}+n_{1}\mathbf{d}_{1})
\\
&=h\mathbf{a}_{1}+n_{1}^{'}\mathbf{d}_{1}+(s_{1}+t_{1})\mathbf{d}_{1}
+\sum_{i\in \{s_{1},t_{1}\}}(\lambda_{i}-1)(h\mathbf{a}_{1}+i\mathbf{d}_{1})
+
\\
&\sum_{i \in [1,n_{1}-1]\setminus \{s_{1},t_{1}\}}\lambda_{i}(h\mathbf{a}_{1}+i\mathbf{d}_{i})+\sum_{j=1}^{n_{2}-1}\mu_{j}(h\mathbf{a}_{2}+j\mathbf{d}_{2}).
\end{align*}
\medskip

\item [(iii)]$s_{2}+t_{2} \leq n_{2}$, then $\mathbf{b}-\mathbf{a}_{2}=\sum_{i=1}^{n_{1}-1}\lambda_{i}(h\mathbf{a}_{1}+i\mathbf{d}_{1})+h\mathbf{a}_{2}+(s_{2}+t_{2})\mathbf{d}_{2}+\sum_{j\in \{s_{2},t_{2}\}}\mu_{j}(h\mathbf{a}_{2}+j\mathbf{d}_{2})+\sum_{j \in [1,n_{2}-1]\setminus \{s_{2},t_{2}\}}\mu_{j}(h\mathbf{a}_{2}+j\mathbf{d}_{2})$.
\medskip

\item[(iv)]$s_{2}+t_{2} >n_{2}$, then $s_{2}+t_{2}=n_{2}+n_{2}^{'}$, and it follows that 
\begin{align*}
&\mathbf{b}-(h\mathbf{a}_{2}+n_{2}\mathbf{d}_{2})\\
&=\sum\limits_{i=1}^{n_{1}-1}\lambda_{i}(h\mathbf{a}_{1}+i\mathbf{d}_{1})+h\mathbf{a}_{2}+n_{2}^{'}\mathbf{d}_{2}+
\\
& \quad \sum\limits_{j\in \{s_{2},t_{2}\}}\mu_{j}(h\mathbf{a}_{2}+j\mathbf{d}_{2})+\sum\limits_{j \in [1,n_{2}-1]\setminus \{s_{2},t_{2}\}}\mu_{j}(h\mathbf{a}_{2}+j\mathbf{d}_{2}).
\end{align*}
\end{itemize}
\smallskip

\noindent In all the above cases, we get a contradiction to $\mathbf{b} \in \mathrm{Ap}(\mathfrak{S}_{2},E_{c})$. 
Hence, 
$\mathrm{Ap}(\mathfrak{S}_{2},E_{c}) \subseteq \{\lambda (h\mathbf{a}_{1}+i\mathbf{d}_{1})+ \mu (h\mathbf{a}_{2}+j\mathbf{d}_{2}) \vert \lambda,\mu\in \{0,1\}, 1 \leq i \leq n_{1}-1,1 \leq j \leq n_{2}-1\}.$
It is easy to check that $\{\lambda (h\mathbf{a}_{1}+i\mathbf{d}_{1})+ \mu (h\mathbf{a}_{2}+j\mathbf{d}_{2}) \vert \lambda,\mu\in \{0,1\}, 1 \leq i \leq n_{1}-1,1 \leq j \leq n_{2}-1\} \subseteq \mathrm{Ap}(\mathfrak{S}_{2},E_{c})$.
\end{proof}

\begin{theorem}\label{Quasi Frobenius}
The Quasi Frobenious set is 
$$QF(\Gamma^{2}_{\mathbf{a},\mathbf{d},h})=\{-(h\mathbf{a}+\mathbf{d}),\dots,-(h\mathbf{a}+(n-1)\mathbf{d}\}.$$ 
\end{theorem}
\begin{proof}
From Theorem \ref{Apery}, we have $\mbox{max}_{\prec}\mathrm{Ap}(\Gamma^{2}_{\mathbf{a},\mathbf{d},h},E)=\{h\mathbf{a}+\mathbf{d},h\mathbf{a}+2\mathbf{d},\dots,h\mathbf{a}+(n-1)\mathbf{d}\}$. 
\end{proof}

\begin{lemma}\label{normal}
$\Gamma^{2}_{\mathbf{a},\mathbf{d},h}$ is a normal semigroup.
\end{lemma}

\begin{proof}
From Theorem \ref{Quasi Frobenius}, we have $-\mathrm{QF}(\Gamma^{2}_{\mathbf{a},\mathbf{d},h})=\{h\mathbf{a}+\mathbf{d},h\mathbf{a}+2\mathbf{d},\dots,h\mathbf{a}+(n-1)\mathbf{d}\}$. Now, for $1 \leq i \leq (n-1), h\mathbf{a}+i\mathbf{d}=(h-\frac{ih}{n})\mathbf{a}+\frac{i}{n}(h\mathbf{a}+n\mathbf{d})$. Hence $-\mathrm{QF}(\Gamma^{2}_{\mathbf{a},\mathbf{d},h})\subseteq \mathrm{Ap}(\Gamma^{2}_{\mathbf{a},\mathbf{d},h},E)$ and by \cite{Jafari-Type}, Theorem 4.6, $\Gamma^{2}_{\mathbf{a},\mathbf{d},h}$ is a normal semigroup.
\end{proof}

\begin{theorem}
The ring $\mathbb{K}[\Gamma^{2}_{\mathbf{a},\mathbf{d},h}]$ is Cohen-Macaulay.  
\end{theorem}
\begin{proof}
Easily follows from Lemma \ref{normal} and Theorem 6.3.5 in \cite{Bruns-Herzog}.
\end{proof}

\begin{theorem}\label{normal-conctanation}
$\mathfrak{S}_{2}$ is a normal semigroup. Moreover, the ring $\mathbb{K}[\mathfrak{S}_{2}]$ is Cohen-Macaulay.
\end{theorem}

\begin{proof}
From Theorem \ref{Quasi Frobenius}, we have $-\mathrm{QF}(\mathfrak{S}_{2})=\{\lambda (h\mathbf{a}_{1}+i\mathbf{d}_{1})+ \mu (h\mathbf{a}_{2}+j\mathbf{d}_{2}) \vert \lambda,\mu\in \{0,1\}, 1 \leq i \leq n_{1}-1,1 \leq j \leq n_{2}-1\}$. Now, for $1 \leq i \leq (n_{1}-1),1 \leq j \leq n_{2}-1, \lambda (h\mathbf{a}_{1}+i\mathbf{d}_{1})+ \mu (h\mathbf{a}_{2}+j\mathbf{d}_{2})=\lambda\big((h-\frac{ih}{n_{1}})\mathbf{a}_{1}+\frac{i}{n_{1}}(h\mathbf{a}_{1}+n_{1}\mathbf{d}_{1})\big)+\mu\big((h-\frac{jh}{n_{2}})\mathbf{a}_{2}+\frac{j}{n_{2}}(h\mathbf{a}_{2}+n_{2}\mathbf{d}_{2})\big)$. Hence $-\mathrm{QF}(\Gamma^{2}_{\mathbf{a},\mathbf{d},h})\subseteq \mathrm{Ap}(\mathfrak{S}_{2},E)$ and by \cite{Jafari-Type}, Theorem 4.6, $\mathfrak{S}_{2}$ is a normal semigroup.
Cohen-Macaulayness of $\mathbb{K}[\mathfrak{S}_{2}]$ follows from \cite[Theorem 6.3.5]{Bruns-Herzog}.
\end{proof}

\begin{theorem}[\cite{Bhardwaj-arithmetic}]\label{Gastinger}
Let $\Gamma$ be a Cohen–Macaulay simplicial affine semigroup and $J$ an ideal of $\mathbb{K}[x_{1},\dots,x_{n}
]$ such that $J\subset I_{\Gamma}$. If $\mathrm{dim}_{\mathbb{K}}\frac{\mathbb{K}[x_{1},\dots,x_{n}]}{J+\langle x_{1},\dots,x_{r}\rangle}=\vert \mathrm{Ap}(\Gamma,E)\vert$, then $J=I_{\Gamma}$.
\end{theorem}

\begin{theorem}\label{Eagon-resolution}
Let $I_{\Gamma^{2}_{\mathbf{a},\mathbf{d},h}}$ be a defining ideal of $\mathbb{K}[\Gamma^{2}_{\mathbf{a},\mathbf{d},h}]$. Then $I_{\Gamma}=I_{2}(P)$, 
where $P=\begin{pmatrix}
x_{1}^{h} &x_{2}&\cdots&x_{n}\\
x_{2}&x_{3}&\cdots&x_{n+1}
\end{pmatrix}.$
A minimal graded resolution of $I_{\Gamma^{2}_{\mathbf{a},\mathbf{d},h}}$  is given by the 
Eagon-Northcott complex.

\end{theorem}
\begin{proof}
It is clear that $I_{2}(P)\subseteq \mathrm{ker}(\phi)$. Consider the ideal 
$$J= I_{2}(P)+\langle x_{1},x_{n+1}\rangle =\big\langle\cup_{j=3}^{n}\{x_{l}x_{j}\vert 2 \leq l \leq j-1\}\cup \{x_{l}^{2}\vert 2\leq l \leq n\}\big\rangle 
+  \langle x_{1},x_{n+1} \rangle.$$
\noindent Then, $\frac{A}{J}=\frac{\mathbb{K}[x_{2},\dots,x_{n}]}{J'}$, 
where $J'=\big\langle\cup_{j=3}^{n}\{x_{l}x_{j}\vert 2 \leq l \leq j-1\}\cup \{x_{l}^{2}\vert 2\leq l \leq n\}\big\rangle$. This shows that $\{1,x_{2},\dots,x_{n-1}\}$ gives a $\mathbb{K}$-basis of $\frac{\mathbb{K}[x_{2},\dots,x_{n}]}{J'}$. From Theorem \ref{Apery}, $\vert \mathrm{Ap}(\Gamma^{2}_{\mathbf{a},\mathbf{d},h},E)\vert=n$. By Theorem \ref{Gastinger}, we have $I_{2}(P)=\mathrm{ker}(\phi)$.
\medskip

Let us consider the set of polynomials $T=\{-x_{i}^{2}+x_{i-1}x_{i+1}\mid 3\leq i\leq n\}\cup\{x_{1}^{h}x_{n+1}-x_{2}x_{n}\}\subset I_{\Gamma^{2}_{\mathbf{a},\mathbf{d},h}}$; 
we claim that $T$ forms a regular sequence in $ \mathbb{K}[x_{1},\dots,x_{n+1}]$. We consider the the degree reverse lexicographic ordering induced by $x_{n+1}>x_{n}>\dots>x_{3}>x_{1}>x_{2}$ and with respect to 
this monomial order $\mathrm{LT}(T)=\{x_{3}^{2},\dots,x_{n}^{2},x_{1}^{h}x_{n+1}\}$. Since elements of 
$\mathrm{LT}(T)$ are the mutually coprime by \cite[Lemma 2.2]{Saha-Regular}, the set $T$ forms a regular 
sequence in $\mathbb{K}[x_{1},\dots,x_{n+1}]$. Since $T\subset I_{\Gamma^{2}_{\mathbf{a},\mathbf{d},h}}$, 
we have  $n-1\leq \mathrm{grade}\,( I_{\Gamma^{2}_{\mathbf{a},\mathbf{d},h}})\leq \mathrm{ht}\, (I_{\Gamma^{2}_{\mathbf{a},\mathbf{d},h}}) \leq (n+1-2)(2+1-2)=n-1$. Therefore, by Theorem \ref{Eagon-Northcott}, 
a minimal graded resolution of $I_{\Gamma^{2}_{\mathbf{a},\mathbf{d},h}}$  is given by the Eagon-Northcott 
complex of $P$.
\end{proof}

\begin{theorem}\label{Grobner basis}
The set $G=\{x_{i+1}x_{j}-x_{i}x_{j+1},x_{2}x_{l}-x_{1}^{h}x_{l+1} \mid 2 \leq i \leq n-2, i+1 \leq j \leq n-1, 2\leq l \leq n-1\}$ forms a Gr\"{o}bner basis of the defining ideal 
$I_{\Gamma^{2}_{\mathbf{a},\mathbf{d},h}}$, with respect to the negative degree 
reverse lexicographic ordering induced by $x_{2}>x_{3}>\dots>x_{n}>x_{1}>x_{n+1}$. 
\end{theorem}

\begin{proof}
We consider the $S$-polynomial of every distinct pair of elements of $G$ and show that it reduces 
to zero after division by $G$. Let $f_{l}=x_{2}x_{l}-x_{1}^{h}x_{l+1}, f_{i,j}=x_{i+1}x_{j}-x_{i}x_{j+1}$ 
and $\mathrm{LM}(f_{l})=x_{2}x_{l}, \mathrm{LM}(f_{i,j})= x_{i+1}x_{j}$.

\begin{enumerate}[(1)]
\item If $l\neq i+1,l=j$;
 \begin{align*}
 S(f_{l},f_{i,j})=&x_{i+1}(x_{2}x_{l}-x_{1}^{h}x_{l+1})-x_{2}(x_{i+1}x_{l}-x_{i}x_{l+1})\\
 =& x_{2}x_{i}x_{l+1}-x_{1}^{h}x_{l+1}x_{i+1}\\
 =& x_{l+1}(x_{2}x_{i}-x_{1}^{h}x_{i+1})\\
 =& x_{l+1}f_{i}.
 \end{align*}
 
\item If $l\neq i+1, l\neq j$, 
then $\mathrm{gcd}(\mathrm{LM}(f_{l}),\mathrm{LM}(f_{i,j}))=1$. Hence, $S(f_{l},f_{i,j})$ reduces to zero.  

\item  If $l= i+1, l= j$;
\begin{align*}
 S(f_{l},f_{i,j})=&x_{l}(x_{2}x_{l}-x_{1}^{h}x_{l+1})-x_{2}(x_{l}^{2}-x_{i}x_{j+1})\\
 =&x_{2}x_{i}x_{j+1}-x_{1}^{h}x_{l}x_{l+1}\\
 =& x_{j+1}(x_{2}x_{i}-x_{1}^{h}x_{i+1})=x_{j+1}f_{i}.
\end{align*}

\item  If $l= i+1, l\neq j$;
\begin{align*}
 S(f_{l},f_{i,j})=&x_{j}(x_{2}x_{l}-x_{1}^{h}x_{l+1})-x_{2}(x_{l}x_{j}-x_{i}x_{j+1})\\
 =& x_{2}x_{i}x_{j+1}-x_{1}^{h}x_{l+1}x_{j}\\
 =& x_{j+1}(x_{2}x_{i}-x_{1}^{h}x_{i+1})-x_{1}^{h}(x_{l+1}x_{j}-x_{l}x_{j+1}).
 \end{align*}
 
\item  If $l_{1}\neq l_{2}$; 
\begin{align*}
 S(f_{l_{1}},f_{l_{2}})=&x_{l_{2}}(x_{2}x_{l_{1}}-x_{1}^{h}x_{l_{1}+1})-x_{l_{1}}(x_{2}x_{l_{2}}-x_{1}^{h}x_{l_{2}+1})\\
 =& x_{1}^{h}(x_{l_{1}}x_{l_{2}+1}-x_{l_{1}+1}x_{l_{2}}).
\end{align*}

\item  Let $f_{i_{1},j_{1}}=(x_{i_{1}+1}x_{j_{1}}-x_{i_{1}}x_{j_{1}+1}),
 f_{i_{2},j_{2}}=(x_{i_{2}+1}x_{j_{2}}-x_{i_{2}}x_{j_{2}+1})$. 
 
\medskip

\begin{enumerate}[(a)]
\item If $i_{1}\neq i_{2}, j_{1}\neq j_{2}$;
\begin{align*}
\quad \quad
S(f_{i_{1},j_{1}},f_{i_{2},j_{2}})=&x_{i_{2}}(x_{i_{1}+1}x_{j_{1}}-x_{i_{1}}x_{j_{1}+1})-x_{i_{1}}(x_{i_{2}+1}x_{j_{2}}-x_{i_{2}}x_{j_{2}+1})\\
=&x_{j_{1}}(x_{i_{1}+1}x_{i_{2}}-x_{i_{1}}x_{i_{2}+1}).
 \end{align*}
 
\item  If $i_{1}\neq i_{2}, j_{1}\neq j_{2}$, then
$\mathrm{gcd}(\mathrm{LM}(f_{i_{1},j_{1}},\mathrm{LM}(f_{i_{2},j_{2}})=1$. Hence $S(f_{i_{1},j_{1}},f_{i_{2},j_{2}})$ reduces to zero. 

\item If $i_{1}= i_{2}, j_{1}\neq j_{2}$, then
\begin{align*}
\quad \quad S(f_{i_{1},j_{1}},f_{i_{2},j_{2}})=&x_{j_{2}}((x_{i_{1}+1}x_{j_{1}}-x_{i_{1}}x_{j_{1}+1})-x_{j_{1}}(x_{i_{2}+1}x_{j_{2}}-x_{i_{2}}x_{j_{2}+1})\\
=& x_{i_{1}}(x_{j_{1}}x_{j_{2}+1}-x_{j_{2}}x_{j_{1}+1}).
\end{align*} 
\end{enumerate}

\end{enumerate}
\noindent Each $S$-polynomial reduces to zero, therefore, by Buchberger’s Criterion, $G=\{x_{i+1}x_{j}-x_{i}x_{j+1},x_{2}x_{l}-x_{1}^{h}x_{l+1} \mid 2 \leq i \leq n-2, i+1 \leq j \leq n-1, 2\leq l \leq n-1\}$ forms a Gr\"{o}bner basis of the defining ideal $I_{\Gamma^{2}_{\mathbf{a},\mathbf{d},h}}$, 
with respect to the negative degree reverse lexicographic ordering induced by 
$x_{2}>x_{3}>\dots>x_{n-1}>x_{1}>x_{n}$.
\end{proof}

\begin{corollary}
The minimal generating set of the initial ideal of $I_{\Gamma^{2}_{\mathbf{a},\mathbf{d},h}}$, 
with respect to the monomial ordering defined in Theorem \ref{Grobner basis}, is given by the 
set 
$G(in_{<}(I_{\Gamma^{2}_{\mathbf{a},\mathbf{d},h}})=\{x_{i+1}x_{j},x_{2}x_{l} \mid 2 \leq i \leq n-2, i+1 \leq j \leq n-1, 2\leq l \leq n-1\}$.
\end{corollary}
\begin{proof}
Easily follows from Theorem \ref{Grobner basis}.
\end{proof}

\begin{definition}
\begin{enumerate}[(i)]
\item {\rm Define
$\mathrm{supp}(T)=\{i \mid z_{i} \text{ divides $m$, for some } m \in T\}$, where $(\emptyset\neq )T\subset \mathrm{Mon}(A)$, $A=\mathbb{K}[z_{1},\dots,z_{r+n}]$ and $\mathrm{Mon}(A)$ is a set of monomials in $A$. }
\item {\rm Ideals $I$ and $J$ of $A$ intersect transversally  if and only if $I \cap J=IJ$.}
\end{enumerate}
\end{definition}

\begin{theorem}\label{Property-Semigroup}
Let $I_{\mathfrak{S}_{2}}$ be a defining ideal of $\mathbb{K}[\mathfrak{S}_{2}]$. 
\begin{enumerate}
\item $I_{\mathfrak{S}_{2}}=I_{2}(P_{1})+I_{2}(P_{2})$, where 
$$P_{1}=\begin{pmatrix}
x_{1}^{h} &x_{2}&\cdots&x_{n_{1}}\\
x_{2}&x_{3}&\cdots&x_{n_{1}+1}
\end{pmatrix}, \quad P_{2}=\begin{pmatrix}
x_{n_{1}+2}^{h} &x_{n_{1}+3}&\cdots&x_{n_{2}}\\
x_{n_{1}+3}&x_{n_{1}+4}&\cdots&x_{n_{2}+1}
\end{pmatrix}.$$

\item Maximal minors of the matrices 
$P_{1}$ and $P_{2}$ form a minimal Gr\"{o}bner basis of $I_{\mathfrak{S}_{2}}$, 
with respect to the negative degree lexicographic ordering induced by 
$x_{2}>x_{3}>\dots>x_{n_{1}}>x_{n_{1}+3}>\dots >x_{n_{2}} >x_{1}>x_{n_{1}+1}>x_{n_{1}+2}>x_{n_{2}+1}$.

\item $\mathbb{K}[\mathfrak{S}_{2}]$ is Koszul.

\item The tensor product of the 
Eagon-Northcott complexes of $P_{1}$ and $P_{2}$ give a minimal graded free 
resolution of $I_{\mathfrak{S}_{2}}$
\end{enumerate} 
\end{theorem}

\begin{proof}

\noindent (1) The proof Follows from Theorem \ref{Generator-Concatanation}. Let us 
give an another proof. It is clear that $I_{2}(P_{1})+I_{2}(P_{2})\subseteq I_{\mathfrak{S}_{2}}$. Consider 
the ideal 
\begin{align*}
J=&I_{2}(P_{1})+I_{2}(P_{2}) +\langle x_{1},x_{n_{1}+1},x_{n_{1}+2},x_{n_{2}+1} \rangle\\
=&\big\langle\cup_{j=3}^{n_{1}}\{x_{l}x_{j}\vert 2 \leq l \leq j-1\}\cup \{x_{l}^{2}\vert 2\leq l \leq n_{1}\}
 \big\rangle + \langle x_{1},x_{n_{1}+1} \rangle +  \\
& \big\langle\cup_{j=n_{1}+3}^{n_{2}}\{x_{l}x_{j}\vert 2 \leq l \leq j-1\}\cup \{x_{l}^{2}\vert (n_{1}+3)\leq l \leq n_{2}\} \big\rangle + \langle x_{n_{1}+2},x_{n_{2}+1} \rangle.
\end{align*}
Then, $\{1,x_{2},\dots,x_{n_{1}},x_{n_{1}+3},\dots,x_{n_{2}}\}\cup \{x_{l_{1}}x_{l_{2}}\vert 2 \leq l_{1} \leq n_{1}, n_{1}+3 \leq l_{2} \leq n_{2}\}$ is a $\mathbb{K}$-basis of 
$\frac{\mathbb{K}[x_{2},\dots,x_{n}]}{J'}$,
and hence $ \mathrm{dim}_{\mathbb{K}}\frac{\mathbb{K}[x_{1},\dots,x_{n}]}{J+\langle x_{1},\dots,x_{r}\rangle}=\vert \mathrm{Ap}(\Gamma_2,E_{c})\vert$. By Theorem \ref{Apery-Concatanation} and Theorem \ref{Gastinger}, we have $I_{2}(P_{1})+I_{2}(P_{2})= I_{\mathfrak{S}_{2}}$.
\medskip

\noindent (2) Since the set of maximal minors of $P_{1}, P_{2}$ are disjoint, using Theorem \ref{Grobner basis}, 
we can prove that they form a minimal Gr\"{o}bner basis with respect to the said monomial order. 
Therefore, we have the required Gr\"{o}bner basis with respect to given monomial ordering.   
\medskip

\noindent (3) By Theorem \ref{Grobner basis}, the initial ideal of $I_{\mathfrak{S}_{2}}$, 
with respect to the negative degree reverse lexicographic ordering induced by 
$x_{1}>\dots>x_{n_{1}+1}>x_{n_{1}+2}>\dots>x_{n_{2}+1}$, is given by 
$in_{<}(I_{\mathfrak{S}_{2}})= \big\langle \{x_{i_{1}}x_{j_{1}} \vert 2 \leq i_{1} \leq n_{1}, i_{1} \leq j_{1} \leq n_{1}\}\cup \{x_{i_{2}}x_{j_{2}} \vert n_{1}+3 \leq i_{2} \leq n_{2}, i_{2} \leq j_{2} \leq n_{2}\}\big\rangle$.
Since all the element of $in_{<}(I_{\mathfrak{S}_{2}})$ are of degree $2$ with respect to 
the standard grading, therefore, by \cite[Corollary 3.14]{Conca-Determinantal}, 
we have the Koszul property of $\mathbb{K}[\mathfrak{S}_{2}]$.
\medskip

\noindent (4) We note that the proof follows from Theorem \ref{Generator-Concatanation}. 
Here we give an another proof using transversal intersection. By the proof of Theorem 
\ref{Property-Semigroup}(3), we have $\mathrm{supp}(in_{<}(I_{2}(P_{1}))=\{i_{1},j_{1} \vert 2 \leq i_{1} \leq n_{1}, i_{1} \leq j_{1} \leq n_{1}\}$ and  $\mathrm{supp}(in_{<}(I_{2}(P_{2}))=\{i_{2},j_{2} \vert n_{1}+3 \leq i_{2} \leq n_{2}, i_{2} \leq j_{2} \leq n_{2}\}$.  
Hence,  $\mathrm{supp}(in_{<}(I_{2}(P_{1}))$ and $\mathrm{supp}(in_{<}(I_{2}(P_{2}))$ are 
disjoint. By \cite[Theorem 2.2]{Saha-Transversal}, $I_{2}(P_{1})$ and $I_{2}(P_{2})$ intersect 
transversally. Using Theorem \ref{Eagon-resolution}, the minimal graded resolution of $I_{2}(P_{i})$ 
is given by the Eagon-Northcott complex of $P_{i}$, say $\mathbb{M}_{i}$, for $i=1,2$. 
By \cite[Corollary 2.3]{Saha-Transversal}, $\mathbb{M}_{1}\otimes \mathbb{M}_{2}$ gives the minimal graded resolution of $I_{2}(P_{1})+I_{2}(P_{2})=I_{\mathfrak{S}_{2}}$.
\end{proof}

\begin{theorem}
The associated graded ring $\mathrm{gr}_{\mathfrak{m}}(\mathbb{K}[\mathfrak{S}_{2}])=\oplus_{i\geq 0} \frac{\mathfrak{m}_{i}}{\mathfrak{m}_{i+1}}$ is Cohen-Macaulay. Moreover, $\beta_{i}(\mathrm{gr}_{\mathfrak{m}}(\mathbb{K}[\mathfrak{S}_{2}]))=\beta_{i}(\mathbb{K}[\mathfrak{S}_{2}])$, for all $i \geq 0$.  
The Castelnuovo–Mumford regularity of $\mathrm{gr}_{\mathfrak{m}}(\mathbb{K}[\mathfrak{S}_{2}])$ is $2$.
\end{theorem}

\begin{proof}
By the proof of Theorem \ref{Property-Semigroup}(3), we see that the variable corresponding to 
the extremal rays $x_{1},x_{n_{1}+1},x_{n_{1}+2},x_{n_{2}+1}$ do not divide any element of 
$G(in_{<}(I_{\mathfrak{S}_{2}}))$. Therefore, by \cite{Associated-Saha}, 
$\mathrm{gr}_{\mathfrak{m}}(\mathbb{K}[\mathfrak{S}_{2}])$ is Cohen-Macaulay. From Theorem 
\ref{Apery}, every element of $\mathrm{Ap}(\mathfrak{S}_{2})$ has a unique expression and 
hence $\mathfrak{S}_{2}$ is homogeneous. By \cite[Theorem 4.7]{Jafari-Type} 
$\beta_{i}(\mathrm{gr}_{\mathfrak{m}}(\mathbb{K}[\mathfrak{S}_{2}]))=\beta_{i}(\mathbb{K}[\mathfrak{S}_{2}])$, 
for all $i \geq 0$. 

By Theorem \ref{Apery}, $\mathrm{max}\{(\mathrm{ord}_{\mathfrak{S}_{2}}(w) \vert w \in \mathrm{Ap}(\mathfrak{S}_{2})\}=2$ . Therefore, by \cite[Proposition 5.13]{Jafari-Reduction}, 
$\mathrm{reg}(\mathrm{gr}_{\mathfrak{m}}(\mathbb{K}[\mathfrak{S}_{2}]))=\mathrm{max}\{\mathrm{ord}_{\mathfrak{S}_{2}}(w) \vert w \in \mathrm{Ap}(\mathfrak{S}_{2})\}=2$.
\end{proof}

\end{document}